\documentclass[11pt]{article}
\usepackage{amsmath,amsthm,amssymb,amsfonts}
\usepackage[numbers,sort&compress]{natbib}
\usepackage{color,colordvi}
\usepackage{soul}
\usepackage{fullpage}

\usepackage{amsmath,amsfonts,amsthm,mathtools,bm}
\usepackage{graphicx}
\usepackage[colorlinks=true, allcolors=blue]{hyperref}
\usepackage[capitalise]{cleveref}
\usepackage{standalone}
\usepackage{indentfirst}
\usepackage[affil-it]{authblk}
\usepackage{diagbox}
\usepackage{changes}
\usepackage{enumitem}
\usepackage{bbm}
\usepackage{orcidlink}

\newtheorem{theorem}{Theorem}[section]
\newtheorem{lemma}[theorem]{Lemma}
\newtheorem{corollary}[theorem]{Corollary}

\newtheorem{conjecture}[theorem]{Conjecture}

\newtheorem{definition}[theorem]{Definition}

\newtheorem*{claim*}{Claim}

\newcommand{\textotherwise}{\text{otherwise}}

\newcommand{\RR}{\mathbb{R}}

\newcommand{\ZZ}{\mathbb{Z}}

\newcommand{\cT}{\mathcal{T}}

\DeclareMathOperator{\diag}{\mathrm{diag}}

\DeclarePairedDelimiter{\card}{\lvert}{\rvert}
\DeclarePairedDelimiter{\set}{\lbrace}{\rbrace}

\DeclarePairedDelimiter{\lrangle}{\langle}{\rangle}

\DeclarePairedDelimiter{\ceil}{\lceil}{\rceil}

\title{Comparison between the first Steklov eigenvalue and algebraic connectivity on trees}

\author{Huiqiu Lin\footnote{email: huiqiulin@126.com}}
\author{Da Zhao\footnote{email: zhaoda@ecust.edu.cn}~\orcidlink{0000-0002-9582-0778}}
\affil{School of Mathematics, East China University of Science and Technology, 130 Meilong Road, Shanghai 200237, China.}

\date{}

\begin{document}
\maketitle

\begin{abstract} 
    Trees can be regarded as discrete analogue of Hadamard manifolds, namely simply-connected Riemannian manifolds of non-positive sectional curvature. 
    In this paper, we compare the first (non-trivial) Steklov eigenvalue and algebraic connectivity of trees with prescribed number of boundary vertices and matching number.
    It is particularly noteworthy that while the extremal trees coincide for both operators, their corresponding eigenvalues differ significantly.
\end{abstract}


\section{Introduction}

Steklov~\cite{stekloff1902ProblemesFondamentauxPhysique, kuznetsov2014LegacyVladimirAndreevich} introduced the so called Steklov problem over a century ago during the study of liquid sloshing. 
Given a compact Riemannian manifold with boundary, the Steklov operator is defined to be the map which sends the Dirichlet boundary data of a harmonic function on the manifold to its Neumann boundary data. 
The Steklov operator establishes the model for inverse problem of detecting the inside of a body by measuring the data on the boundary. 
The eigenvalues of the Steklov operators are called Steklov eigenvalues. 
The relationship between Steklov eigenvalues and the Yamabe problem on Riemannian manifold with boundary was found by Escobar~\cite{escobar1992YamabeProblemManifolds}. 
In the last decade, Fraser and Scheon~\cite{fraser2011FirstSteklovEigenvalue, fraser2016SharpEigenvalueBounds, fraser2019ShapeOptimizationSteklov, fraser2020ResultsHigherEigenvalue} showed the relationship between Steklov eigenfunctions and free boundary minimal submanifolds in Euclidean ball and proceed on extremal problems of Steklov eigenvalues. 

Independently, Hua--Huang--Wang~\cite{hua2017FirstEigenvalueEstimates} and Hassannezhad--Miclo~\cite{asmahassannezhad2020HigherOrderCheeger} extended the Steklov problem to discrete spaces. 
On the other hand, Colbois--Girouard~\cite{girouard2014SpectralGapGraphs} constructed compact surfaces with first nontrivial Steklov eigenvalue uniformly bounded away from zero from expander graphs. 
Numerous studies investigated the properties of Steklov eigenvalues on graphs later. 
In~\cite{han2023SteklovEigenvalueProblem, asmahassannezhad2020HigherOrderCheeger,hua2017FirstEigenvalueEstimates, hua2023CheegerEstimatesDirichlettoneumann, perrin2019LowerBoundsFirst, perrin2020IsoperimetricUpperBound,tschanz2022UpperBoundsSteklov}, the authors controlled the Steklov eigenvalues by isoperimetrical parameters. 
He--Hua studied the upper bound of first Steklov eigenvalue on trees in~\cite{he2022UpperBoundsSteklov,he2022SteklovFlowsTrees}. 
In~\cite{linFirstSteklovEigenvalue2024}, the authors give an upper bound for the first Steklov eigenvalue on planar graphs and block graphs. 
In~\cite{lin2024UpperBoundsSteklov}, the authors provide an upper bound for the first Steklov eigenvalue on graphs of genus $g$. 
In~\cite{yu2024MonotonicitySteklovEigenvalues}, Yu--Yu proved the monotonicity of Steklov eigenvalues. 
Perrin provided some lower bounds of Steklov eigenvalues in~\cite{perrin2019LowerBoundsFirst}, and Shi--Yu generalized the bound to weighted graphs~\cite{shi2025ExtensionRigidityPerrins}. 
Shi--Yu compared the Steklov eigenvalue and the Laplacian eigenvalue in~\cite{shi2022ComparisonSteklovEigenvalues}. 
In~\cite{shi2022LichnerowicztypeEstimateSteklov}, the authors proved a Lichnerowicz--type estimate for Steklov eigenvalues. 
In~\cite{shi2022DirichlettoneumannMapsDifferential}, the authors established Steklov operators for differential forms on graphs. 
The survey paper \cite{colbois2023RecentDevelopmentsSteklov} summarizes recent progress on Steklov problem.

A graph is a pair $G = (V, E)$, where $V$ is the vertex set and $E \subseteq \binom{V}{2}$ is the edge set.
The edge $(x, y) \in E$ is sometimes abbreviated as $xy$ or $x \sim y$. 
Let $B$ be a subset of the vertex set, which we call the boundary of the graph. 
For a graph $G$ with boundary $B$, we can define the (discrete) Steklov operator $\Lambda: \RR^B \to \RR^B$ by $\Lambda(f) = \frac{\partial \hat{f}}{\partial n}$, where $\hat{f}$ is the harmonic extension of $f$ to the whole graph $V$, and $\frac{\partial \hat{f}}{\partial n}$ is the discrete normal derivative given by $\frac{\partial \hat{f}}{\partial n}(x) = \sum_{(x, y) \in E} (f(x) - f(y))$. 
The eigenvalues of $\Lambda$ are called Steklov eigenvalues of the pair $(G, B)$, denoted by 
\begin{align}
    0 = \sigma_1(G,B) \leq \sigma_2(G,B) \leq \cdots \leq \sigma_{\card{B}}(G, B). 
\end{align}
We write $\sigma_i$ for $\sigma_i(G, B)$ when the graph and the boundary are clear from context. 

A path of length $\ell$ in a graph is a sequence of distinct vertices $v_0 v_1 \cdots v_{\ell}$ such that $v_{i-1} \sim v_i$ for all $i = 1,2, \ldots, \ell$. 
The distance between two vertices $x, y \in V$ is the length of shorted path connecting them. 
The diameter $D$ of a graph $G$ is the maximum distance among all pairs of vertices. 
A graph $G$ is connected if for every pair of vertices $x, y \in V$, there exists a path connecting them. 
A tree is a minimal connected graph, namely it is a connected graph with $(\card{V} - 1)$ edges. 
The path graph $P_n$ is a graph consisting of a path with $n$ vertices. 
The degree of a vertex $v \in V$ is the number of edges adjacent to it, namely $\deg(v) = \card{\set{u \in V: u \sim v}}$. 
The vertices of degree $1$ in a tree are called leaves. 
We take the set of leaves as the natural boundary of a tree. 
A matching $M$ on a graph $G = (V, E)$ is a subset of the edges such that no two chosen edges share an end vertex. 
The (maximum) matching number of a graph $G$ is the maximum number of edges among all matchings on the graph $G$.
A perfect matching is a matching which occupies all vertices. 
Hall~\cite{hall1935RepresentativesSubsets} characterized perfect matching for bipartite graphs and Tutte~\cite{tutte1947FactorizationLinearGraphs} characterized perfect matching for general graphs. 

Trees can be regarded as discrete analogue of Hadamard manifolds, namely simply-connected Riemannian manifolds of non-positive sectional curvature. Establishing bounds for the Steklov eigenvalues of trees has attracted significant attention from researchers.
He--Hua studied the upper bound of first Steklov eigenvalue on trees in~\cite{he2022UpperBoundsSteklov,he2022SteklovFlowsTrees}, by using diameter $D$, they obtained a very useful upper bound of $\sigma_2$, that is $$\sigma_2\leq \frac{2}{D}. $$
The authors~\cite{lin2025MaximizeSteklovEigenvalue} determine the maximum Steklov eigenvalue on trees with given vertex number and leaf number or given vertex number and diameter. 
In~\cite{yu2022MinimalSteklovEigenvalues}, the authors characterized the minimal Steklov eigenvalues on trees. 
In this paper, we determine the maximal Steklov eigenvalues on trees with given matching number, which can be regarded as a continuation of the work in~\cite{he2022UpperBoundsSteklov,he2022SteklovFlowsTrees,yu2024MonotonicitySteklovEigenvalues,lin2025MaximizeSteklovEigenvalue}.

Let $n, m$ be integers such that $n \geq 2$ and $1 \leq m \leq n/2$. 
Let $\cT(n, m)$ be the set of trees with $n$ vertices and maximum matching number $m$. 

Our first main result is an upper bound of $\sigma_{2}(T)$ for $T \in \cT(n, m)$.

\begin{theorem}\label{thm:slope}
    Let $n, m$ be integers such that $n \geq 2$ and $1 \leq m \leq n/2$. 
    Let $T$ be a tree with $n$ vertices with maximum matching number $m$. 
    Then 
    \begin{align}
        \sigma_{2}(T) \leq 
        \begin{cases}
            1, & m = 1, \\
            \frac{n-2}{2n-5}, & m = 2, \\
            \frac{1}{2}, & m \geq 3.
        \end{cases}
    \end{align}
    For $m = 1$, the equality $\sigma_2(T) = 1$ holds if and only if $T$ is the star graph $S_n$. 
    For $m = 2$, the equality $\sigma_2(T) = \frac{n-2}{2n-5}$ holds if and only if $T$ is the crab graph $CG_{1, n-3; 1}$. 
    For $m \geq 3$, the equality $\sigma_2(T) = \frac{1}{2}$ holds if and only if $T$ is the spider graph $Sp_{m-1, n-2m+1; 2, 1}$. 
    See~\cref{sec:wide} for the definition of crab graph and spider graph. 
\end{theorem}

Let $b, m, k$ be integers such that $b \geq 2$, $m \geq 1$, and $b \geq k \geq 1$. 
Let $\widetilde{\cT}(b, m)$ be the set of trees with $b$ leaves and maximum matching number $m$. 

Our second main result is an upper bound of $\sigma_{2}(T)$ for $T \in \widetilde{\cT}(b, m)$.

\begin{theorem}\label{thm:fell}
    Let $b, m$ be integers such that $b \geq 2$ and $m \geq 1$. 
    Let $T$ be a tree with $b$ leaves and maximum matching number $m$. 
    Then
    \begin{align}
        \sigma_{2}(T) \leq
        \begin{cases}
            1, & m = 1, \\
            \frac{b}{2b-1}, & m = 2, \\
            \frac{2}{2m-1}, & b=2, m \geq 3, \\
            \frac{1}{2r + \frac{b - 1}{b}}, & b \geq 3, m = br+1, r \in \ZZ_{+}, \\
            \frac{2}{4r+3}, & b \geq 3, m = br+2, r \in \ZZ_{+}, \\
            \frac{1}{2r+2}, & b \geq 3, m = br+s, 3 \leq s \leq b, r \in \ZZ_{\geq 0}. 
        \end{cases}
    \end{align}

    For $m = 1$, the equality $\sigma_2(T) = 1$ holds if and only if $T$ is the star graph $S_{b+1}$. 
    For $m = 2$, the equality $\sigma_2(T) = \frac{b}{2b-1}$ holds if and only if $T$ is the crab graph $CG_{1, b-1; 1}$. 
    For $b = 2, m \geq 3$, the equality $\sigma_2(T) = \frac{2}{2m-1}$ holds if and only if $T$ is the path graph $P_{2m}$. 
    For $b \geq 3, m = br+1, r \in \ZZ_+$, the equality $\sigma_2(T) = \frac{1}{2r+\frac{b-1}{b}}$ holds if and only if $T$ is the crab graph $CG_{1, b-1; 2r} \cong Sp_{1, b-1; 2r+1, 2r}$. 
    For $b \geq 3, m = br+s, 3 \leq s \leq b, r \in \ZZ_{\geq 0}$, the spider graph $Sp_{2r+2, 2r+1; s-1, b-s+1} \in \widetilde{\cT}(b, br+s)$ achieves the equality $\sigma_2(T) = \frac{1}{2r+2}$. 
    The tree achieves the equality may not be unique, say both $Sp_{2r+2, 2r+1; b-1, 1}, Sp_{2r+2; b} \in \widetilde{\cT}(b, br+b)$ achieve the equality $\sigma_2(T) = \frac{1}{2r+2}$. 
\end{theorem}

\begin{conjecture}
    Let $b, r$ be integers such that $b \geq 2$ and $r \geq 1$. 
    Let $T$ be a tree with $b$ leaves and maximum matching number $br+2$. 
    Then $\sigma_{2}(T) \leq \sigma_{2}^-(ES_{b, 2r})$ where 
    \begin{align}
        \sigma_{2}^{\pm}(ES_{b; p}) = \frac{2bp + 3b - 3 \pm \sqrt{b^2-2 b+9}}{2(bp^2 + 3bp - 3p + 2b - 4)}.
    \end{align} 
    The equality holds if and only if $T$ is the extra special graph $ES(b, 2r)$.
    See~\cref{sec:wide} for the definition of extra special graph. 
\end{conjecture}

For higher Steklov eigenvalues, the maximum value is always $1$. 

\begin{theorem}\label{thm:older}
    Let $b, m, k$ be integers such that $b \geq k \geq 3$ and $m \geq 1$. 
    Let $T$ be a tree with $b$ leaves and maximum matching number $m$.
    Then
    \begin{align}
        \sigma_{k}(T) \leq 1.
    \end{align}
    The spider graph $Sp_{1, b-1; 2m-1, 1} \in \widetilde{\cT}(b, m)$ achieves the equality $\sigma_k(T) = 1$.
\end{theorem}

For a graph $G = (V, E)$, we can define the (discrete) Laplacian operator $L: \RR^V \to \RR^V$ by $L(f) = \sum_{xy \in E} (f(x) - f(y))$. 
The eigenvalues of $L$ are called the Laplacian eigenvalues of $G$, denoted by 
\begin{align}
    0 = \lambda_1(G) \leq \lambda_2(G) \leq \cdots \lambda_{|V|}(G). 
\end{align}
We write $\lambda_i$ for $\lambda_i(G)$ when the graph is clear from context. 
The second Laplacian eigenvalue $\lambda_2$ is also known as the algebraic connectivity $a(G)$. 

Our third main result is an upper bound of $\lambda_{2}(T)$ for $T \in \cT(n, m)$.

\begin{theorem}\label{thm:ranch}
    Let $n, m$ be integers such that $n \geq 2$ and $1 \leq m \leq n/2$. 
    Let $T$ be a tree with $n$ vertices with maximum matching number $m$.  
    Then 
    \begin{align}
        \lambda_{2, \text{max}}(n, m) \leq 
        \begin{cases}
            1, & m = 1, \\
            \lambda_2(CG_{1, n-2; 1}), & m = 2, \\
            4 \sin^2 \frac{\pi}{10}, & m \geq 3,
        \end{cases}
    \end{align}
    where $\lambda_2(CG_{1, n-2, 1})$ is the smallest root of $x^3 - (n+3) x^2 + (3n+1) x - n - 1 = 0$.
    For $m = 1$, the equality $\lambda_2(T) = 1$ holds if and only if $T$ is the star graph $S_n$. 
    For $m = 2$, the equality $\sigma_2(T) = \lambda_2(CG_{1, n-2; 1})$ holds if and only if $T$ is the crab graph $CG_{1, n-3; 1}$. 
    For $m \geq 3$, the equality $\sigma_2(T) = 4 \sin^2 \frac{\pi}{10}$ holds if and only if $T$ is the spider graph $Sp_{m-1, n-2m+1; 2, 1}$. 
\end{theorem}

\cref{thm:ranch} can be regarded as a refinement of~\cite[Theorem 1.3]{zhu2013AlgebraicConnectivityGraphs}.

Our second main result is an upper bound of $\lambda_{2}(T)$ for $T \in \widetilde{\cT}(b, m)$.

\begin{theorem}\label{thm:unit}
    Let $b, m$ be integers such that $b \geq 2$ and $m \geq 1$. 
    Let $T$ be a tree with $b$ leaves and maximum matching number $m$. 
    Then
    \begin{align}
        \sigma_{2}(T) \leq
        \begin{cases}
            1, & m = 1, \\
            \lambda_2(CG_{1, b-1, 1}), & m = 2, \\
            4 \sin^2 \frac{\pi}{4m}, & b=2, m \geq 3, \\
            \lambda_2(CG_{1, b-1; 2r}), & b \geq 3, m = br+1, r \in \ZZ_{+}, \\
            4 \sin^2 \frac{\pi}{8r+8}, & b \geq 3, m = br+2, r \in \ZZ_{+}, \\
            4 \sin^2 \frac{\pi}{8r+10}, & b \geq 3, m = br+s, 3 \leq s \leq b, r \in \ZZ_{\geq 0},
        \end{cases}
    \end{align}
    where $\lambda_2(CG_{1, b-1, 1})$ is the smallest root of $x^3 - (b+4) x^2 + (3b+4) x - b - 2 = 0$.

    For $m = 1$, the equality $\sigma_2(T) = 1$ holds if and only if $T$ is the star graph $S_{b+1}$. 
    For $m = 2$, the equality $\sigma_2(T) = \lambda_2(CG_{1, b-1, 1})$ holds if and only if $T$ is the crab graph $CG_{1, b-1; 1}$. 
    For $b = 2, m \geq 3$, the equality $\sigma_2(T) = 4 \sin^2 \frac{\pi}{4m}$ holds if and only if $T$ is the path graph $P_{2m}$. 
    For $b \geq 3, m = br+1, r \in \ZZ_+$, the equality $\sigma_2(T) = \lambda_2(CG_{1,b-1;2r})$ holds if and only if $T$ is the crab graph $CG_{1, b-1; 2r} \cong Sp_{1, b-1; 2r+1, 2r}$. 
    For $b \geq 3, m = br+s, 3 \leq s \leq b, r \in \ZZ_{\geq 0}$, the spider graph $Sp_{s-1, b-s+1; 2r+2, 2r+1} \in \widetilde{\cT}(b, br+s)$ achieves the equality $\sigma_2(T) = 4 \sin^2 \frac{\pi}{8r+10}$. 
    The tree achieves the equality may not be unique, say both $Sp_{b-1, 1; 2r+2, 2r+1}, Sp_{b; 2r+2} \in \widetilde{\cT}(b, br+b)$ achieve the equality $\sigma_2(T) = 4 \sin^2 \frac{\pi}{8r+10}$. 
\end{theorem}

\begin{conjecture}
    Let $b, r$ be integers such that $b \geq 2$ and $r \geq 1$. 
    Let $T$ be a tree with $b$ leaves and maximum matching number $br+2$. 
    Then $\sigma_{2, \text{max}}(T) \leq \lambda_{2}^-(ES_{b, 2r})$ for $b \geq 3, m = br+2, r \in \ZZ_{+}$.
    The equality holds if and only if $T$ is the extra special graph $ES(b, 2r)$.
\end{conjecture}

\section{Proofs}\label{sec:wide}

\subsection{Steklov eigenvalue}

Recall that shrinking a tree only enlarges the Steklov eigenvalues. 

\begin{lemma}[{\cite[Corollary 1.1]{yu2024MonotonicitySteklovEigenvalues}}]\label{lem:monotone}
    Let $T$ be a finite tree with leaves as boundary $B$. 
    Let $T'$ be a subtree of $T$ with leaves as boundary $B'$. 
    Then
    \begin{align}
        \sigma_{i}(T) \leq \sigma_{i}(T')
    \end{align}
    for $i = 1,2, \ldots, \card{B'}$.
\end{lemma}

Let us consider a family of trees with one center. 

\begin{definition}\label{def:spider}
    Let $p_1, p_2, \ldots, p_t \geq 0$, $\ell_1 \geq \ell_2 \geq \cdots \geq \ell_t\geq 1$ be integers such that $p_1+p_2 \geq 2$. 
    Consider $p_i$ paths of length $\ell_i$ for $i = 1, 2, \ldots, t$. 
    For each of the $(p_1 \ell_1 + \cdots + p_t \ell_t)$ paths, take one end of the path and identify these ends as one vertex. 
    We get a graph called the spider graph $Sp_{p_1, p_2, \ldots, p_t; \ell_1, \ell_2, \ldots, \ell_t}$.
    (see \cref{fig:spider}).
\end{definition}

\begin{figure}[htbp]
    \centering
    \begin{tikzpicture}
    [scale=2, every node/.style={circle, inner sep=1pt, minimum size = 2mm}]
    
    \node[draw, label=above:{$u_{1,1}$}] (u11) at (-1, 1) {};
    \node[draw, label=below:{$u_{1,2}$}] (u12) at (-1, 0.8) {};
    \node (u1m) at (-1,0.2) {$\vdots$};
    \node[draw, label=below:{$u_{1,p}$}] (u1p) at (-1, -0.4) {};

    \node[draw, label=above:{$u_{1,1}$}] (u21) at (-2, 1) {};
    \node[draw, label=below:{$u_{1,2}$}] (u22) at (-2, 0.8) {};
    \node (u2m) at (-1,0.2) {$\vdots$};
    \node[draw, label=below:{$u_{1,p}$}] (u2p) at (-2, -0.4) {};

    \node[draw, fill=black, label=left:{$u_{\ell_1,1}$}] (u31) at (-3, 1) {};
    \node[draw, fill=black, label=left:{$u_{\ell_1,2}$}] (u32) at (-3, 0.8) {};
    \node (u3m) at (-3,0.2) {$\vdots$};
    \node[draw, fill=black, label=left:{$u_{\ell_1,p}$}] (u3p) at (-3, -0.4) {};
    
    \node[draw, label=above:{$o$}] (o1) at (0, 0) {};

    \node[draw, label=above:{$v_{1,1}$}] (v11) at (1, 1) {};
    \node[draw, label=below:{$v_{1,2}$}] (v12) at (1, 0.8) {};
    \node (v1m) at (1,0.2) {$\vdots$};
    \node[draw, label=below:{$v_{1,p}$}] (v1p) at (1, -1) {};

    \node[draw, label=above:{$v_{1,1}$}] (v21) at (2, 1) {};
    \node[draw, label=below:{$v_{1,2}$}] (v22) at (2, 0.8) {};
    \node (v2m) at (1,0.2) {$\vdots$};
    \node[draw, label=below:{$v_{1,p}$}] (v2p) at (2, -1) {};

    \node[draw, fill=black, label=right:{$v_{\ell_2,1}$}] (v31) at (3, 1) {};
    \node[draw, fill=black, label=right:{$v_{\ell_2,2}$}] (v32) at (3, 0.8) {};
    \node (v3m) at (3,0.2) {$\vdots$};
    \node[draw, fill=black, label=right:{$v_{\ell_2,p}$}] (v3p) at (3, -1) {};

    \draw (o1) -- (u11);
    \draw (o1) -- (u12);
    \draw (o1) -- (u1m);
    \draw (o1) -- (u1p);
    
    \node (um1) at (-1.5,1) {$\cdots$};
    \node (um2) at (-1.5,0.8) {$\cdots$};
    \node (ump) at (-1.5,-0.4) {$\cdots$};

    \draw (u21) -- (u31);
    \draw (u22) -- (u32);
    \draw (u2p) -- (u3p);
    
    \draw (o1) -- (v11);
    \draw (o1) -- (v12);
    \draw (o1) -- (v1m);
    \draw (o1) -- (v1p);
    
    \node (vm1) at (1.5,1) {$\cdots$};
    \node (vm2) at (1.5,0.8) {$\cdots$};
    \node (vmp) at (1.5,-1) {$\cdots$};

    \draw (v21) -- (v31);
    \draw (v22) -- (v32);
    \draw (v2p) -- (v3p);
    
\end{tikzpicture}
    \caption{The spider graph $Sp_{p_1, q_2; \ell_1, \ell_2}$ with leaves as boundary.}
    \label{fig:spider}
\end{figure}
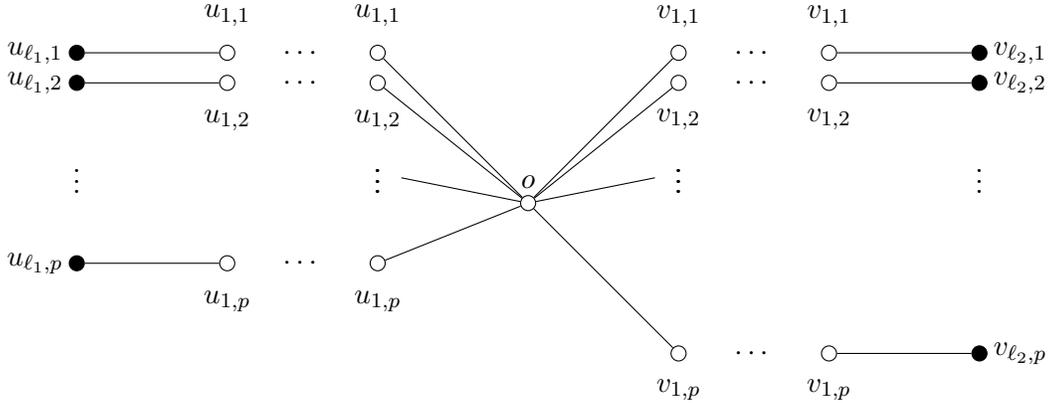

Next we determine the Steklov eigenvalues of the spider graphs with two different lengths. 

\begin{lemma}\label{lem:spider}
    Let $p_1 \geq 0$, $p_2 \geq 0$, $\ell_1 \geq \ell_2 \geq 1$ be integers such that $p_1+p_2 \geq 2$. 
    Let $T$ be the spider graph $Sp_{p_1, p_2; \ell_1, \ell_2}$ with leaves as boundary. 
    Then $\sigma_1(T) = 0$, $\sigma_2(T) = \cdots = \sigma_{p_1}(T) = \frac{1}{\ell_1}$, $\sigma_{p_1+1}(T) = \frac{p_1+p_2}{\ell_2 p_1 + \ell_1 p_2}$, and $\sigma_{p_1+2}(T) = \cdots = \sigma_{p_1+p_2}(T) = \frac{1}{\ell_2}$. 
    If $p_1 = 0$, then the eigenvalue $\frac{p_1+p_2}{\ell_2 p_1 + \ell_1 p_2}$ is missing.  
    If $p_1 = 0$ or $1$, then the eigenvalue $\frac{1}{\ell_1}$ is missing. 
    If $p_2 = 0$, then the eigenvalue $\frac{p_1+p_2}{\ell_2 p_1 + \ell_1 p_2}$ is missing.  
    If $p_2 = 0$ or $1$, then the eigenvalue $\frac{1}{\ell_2}$ is missing.  
\end{lemma}

\begin{proof}
    We give the Steklov eigenfunctions directly. 
    Define $\sigma_1 = 0$, $\sigma_2 = \cdots = \sigma_{p_1} = \frac{1}{\ell_1}$, $\sigma_{p_1+1} = \frac{p_1+p_2}{\ell_2 p_1 + \ell_1 p_2}$, and $\sigma_{p_1+2} = \cdots = \sigma_{p_1+p_2} = \frac{1}{\ell_2}$. 
    Define
    \begin{align}
        \xi_1(v) = 1, \quad \forall v \in V; 
    \end{align}
    \begin{align}
        \xi_m(v) = 
        \begin{cases}
            1 - \frac{\ell_1 - j}{\ell_1}, & v = u_{m-1,j}, 1 \leq j \leq \ell_1, \\
            -(1 - \frac{\ell_1 - j}{\ell_1}), & v = u_{m,j}, 1 \leq j \leq \ell_1, \\
            0, & \textotherwise,
        \end{cases}
    \end{align}
    for $m = 2, 3, \ldots, p_1$;
    \begin{align}
        \xi_{p_1+1}(v) = 
        \begin{cases}
            p_2(1 - \frac{(\ell_1 - j)(p_1+p_2)}{\ell_2 p_1 + \ell_1 p_2}), & v = u_{i,j}, 1 \leq i \leq p_1, 1 \leq j \leq \ell_1, \\
            -p_1(1 - \frac{(\ell_2 - j)(p_1+p_2)}{\ell_2 p_1 + \ell_1 p_2}), & v = v_{i,j}, 1 \leq i \leq p_2, 1 \leq j \leq \ell_2, \\
            \frac{(\ell_2 - \ell_1)p_1 p_2}{\ell_2 p_1 + \ell_1 p_2}, & v = o;
        \end{cases}
    \end{align}
    and
    \begin{align}
        \xi_m(v) = 
        \begin{cases}
            1 - \frac{\ell_2 - j}{\ell_2}, & v = v_{m-q-1,j}, 1 \leq j \leq \ell_2, \\
            -(1 - \frac{\ell_2 - j}{\ell_2}), & v = v_{m-q,j}, 1 \leq j \leq \ell_2, \\
            0, & \textotherwise,
        \end{cases}
    \end{align}
    for $m = p_1+2, p_1+3, \ldots, p_1+p_2$.
    Then $\xi_j$ is a Steklov eigenfunction corresponding to Steklov eigenvalue $\sigma_j$ for $j = 1, 2, \ldots, p_1+p_2$.
\end{proof}

\begin{definition}\label{def:extraspecial}
    Let $p \geq 1$ and $b \geq 3$ be integers.
    The extra special graph $ES_{b; p}$ is the spider graph $Sp_{1, 1, b-2; p+2, p+1, p}$.
\end{definition}

Next we determine the Steklov eigenvalues of the extra special graph.

\begin{lemma}\label{lem:extendedstar}
    Let $p \geq 1$ and $b \geq 3$ be integers. 
    Let $T$ be the extra special graph $ES_{b; p}$. 
    Then $\sigma_1(T) = 0$, $\sigma_2(T) = \sigma_{ES}^-(b; p)$, $\sigma_{3}(T) = \sigma_{ES}^+(b; p)$, and $\sigma_{4}(T) = \cdots = \sigma_{b}(T) = \frac{1}{p}$, 
    where 
    \begin{align}
        \sigma_{ES}^{\pm}(b; p) = \frac{2bp + 3b - 3 \pm \sqrt{b^2-2 b+9}}{2(bp^2 + 3bp - 3p + 2b - 4)}.
    \end{align}
    If $b = 3$, then the eigenvalue $\frac{1}{p}$ is missing. 
\end{lemma}

\begin{proof}
    We give the Steklov eigenfunctions directly. 
    Define $\sigma_1 = 0$, $\sigma_2 = \sigma_{ES}^-(b; p)$, $\sigma_{3} = \sigma_{ES}^+(b; p)$, and $\sigma_{4} = \cdots = \sigma_{b} = \frac{1}{p}$.
    Define
    \begin{align}
        \xi_1(v) = 1, \quad \forall v \in V; 
    \end{align}
    \begin{align}
        \xi_m(v) = 
        \begin{cases}
            f_1(1 - (p+2-j)\sigma_m), & v = v_{1,j}, 1 \leq j \leq p+2, \\
            f_2(1 - (p+1-j)\sigma_m), & v = v_{2,j}, 1 \leq j \leq p+1, \\
            1 - (p-j)\sigma_m, & v = v_{i,j}, 3 \leq i \leq b, 1 \leq j \leq p, \\
            1 - p \sigma_m, & v = v_{0}, 
        \end{cases}
    \end{align}
    where
    \begin{align}
        f_1 &= 1 - 2b - bp + (bp^2 + 3bp - 3p + 2b - 4)\sigma_m, \\
        f_2 &= 1 + b + bp + (bp^2 + 3bp - 3p + 2b - 4)\sigma_m,
    \end{align}
    for $m = 2, 3$;
    \begin{align}
        \xi_m(v) = 
        \begin{cases}
            1 - \frac{p - j}{p}, & v = v_{m-1,j}, 1 \leq j \leq p, \\
            -(1 - \frac{p - j}{p}), & v = v_{m,j}, 1 \leq j \leq p, \\
            0, & \textotherwise,
        \end{cases}
    \end{align}
    for $m = 4, 5, \ldots, b$.
    Then $\xi_j$ is a Steklov eigenfunction corresponding to Steklov eigenvalue $\sigma_j$ for $j = 1, 2, \ldots, b$.
\end{proof}

Next we consider a family of trees with two centers. 

\begin{definition}
    Let $b_1, b_2, r \geq 1$ be integers. 
    The crab graph $CG_{b_1, b_2; r}$ is obtained by identifying the center of the spider graph $Sp_{b_1; r}$ to an end of a separate edge, and identifying the center of another spider graph $Sp_{b_2; r}$ to the other end of the edge. 
    See~\cref{fig:spring}.
\end{definition}

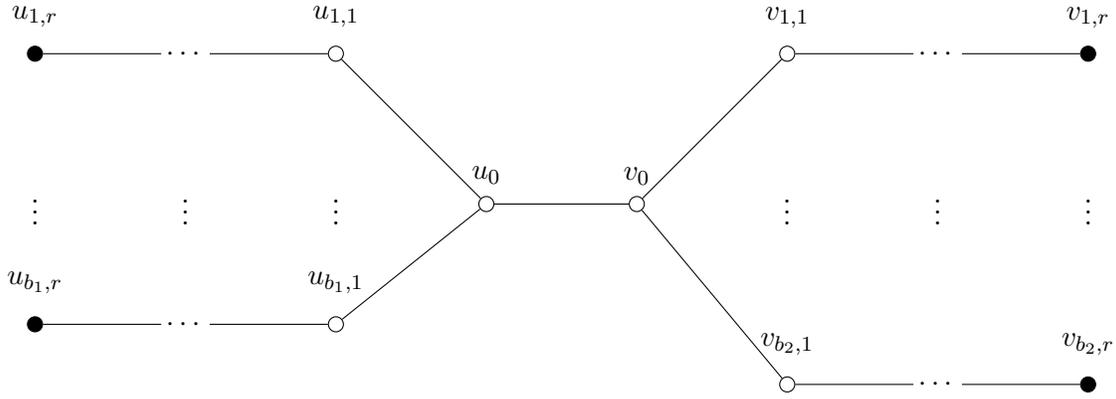
\begin{figure}[htbp]
    \centering
    \begin{tikzpicture}
    [scale=2, every node/.style={circle, inner sep=1pt, minimum size = 2mm}]

    \node[draw, label=above:{$u_0$}] (u0) at (-1, 0) {};
    \node[draw, label=above:{$u_{1,1}$}] (u11) at (-2, 1) {};
    \node (um1) at (-2, 0) {$\vdots$};
    \node[draw, label=above:{$u_{b_1,1}$}] (ub1) at (-2, -0.8) {};


    \node (u13) at (-3, 1) {$\cdots$};
    \node (um3) at (-3, 0) {$\vdots$};
    \node (ub3) at (-3, -0.8) {$\cdots$};

    \node[draw, fill=black, label=above:{$u_{1,r}$}] (u14) at (-4, 1) {};
    \node (um4) at (-4, 0) {$\vdots$};
    \node[draw, fill=black, label=above:{$u_{b_1,r}$}] (ub4) at (-4, -0.8) {};

    \draw (u0) -- (u11) -- (u13) -- (u14);
    \draw (u0) -- (ub1) -- (ub3) -- (ub4);

    \node[draw, label=above:{$v_0$}] (v0) at (0, 0) {};
    \node[draw, label=above:{$v_{1,1}$}] (v11) at (1, 1) {};
    \node (um1) at (1, 0) {$\vdots$};
    \node[draw, label=above:{$v_{b_2,1}$}] (vb1) at (1, -1.2) {};


    \node (v13) at (2, 1) {$\cdots$};
    \node (vm3) at (2, 0) {$\vdots$};
    \node (vb3) at (2, -1.2) {$\cdots$};

    \node[draw, fill=black, label=above:{$v_{1,r}$}] (v14) at (3, 1) {};
    \node (um4) at (3, 0) {$\vdots$};
    \node[draw, fill=black, label=above:{$v_{b_2,r}$}] (vb4) at (3, -1.2) {};

    \draw (v0) -- (v11) -- (v13) -- (v14);
    \draw (v0) -- (vb1) -- (vb3) -- (vb4);

    \draw (u0) -- (v0);

\end{tikzpicture}
    \caption{a crab graph $CG_{b_1, b_2; r}$ with leaves as boundary.}\label{fig:spring}
\end{figure}

The Steklov eigenvalue of crab graphs are given as follows.

\begin{lemma}[{\cite[Lemma 2.7]{lin2025MaximizeSteklovEigenvalue}}]
    Let $b_1, b_2, r \geq 1$ be integers. 
    Let $T = (V, E)$ be the crab graph $CG_{b_1, b_2; r}$ with leaves as boundary $B$. 
    Then the Steklov eigenvalues of $T$ are
    $\sigma_1 = 0$, $\sigma_2 = \frac{b_1 + b_2}{b_1 b_2 + r(b_1 + b_2)}$, and $\sigma_3 = \cdots = \sigma_{b_1 + b_2} = 1/r$. 
\end{lemma}

We are prepared to prove~\cref{thm:slope}.

\begin{proof}[Proof of~\cref{thm:slope}]
    Let $T$ be a tree with $n$ vertices and matching number $m$. 
    If $T$ contains a path of length $\ell$, then the matching number of $T$ is at least $\ceil{\ell/2}$. 
    Therefore, if the matching number of a tree is $m$, then the diameter $D$ of the tree is at most $2m$. 

    If $m = 1$, then the diameter is at most $2$. 
    The only tree with diameter $1$ is the star graph $S_2$. 
    Every tree with diameter $2$ is a star graph. 
    It is clear that the second Steklov eigenvalue of a star graph is $1$. 

    If $m = 2$, then the diameter is at most $4$. 
    If the diameter of $T$ is $4$, then it contains the path graph $P_5$ as a subtree. 
    Then by~\cref{lem:monotone} $\sigma_2(T) \leq \sigma_2(P_5) = \frac{1}{2} < \frac{n-2}{2n-5}$. 
    If the diameter of $T$ is $3$, then $T$ is a crab graph $CG_{p,q;1}$ with $p + q = n-2$. 
    Then $\sigma_2(T) = \frac{p+q}{pq + p + q} = \frac{n-2}{p(n-2-p) + n-2}\leq \frac{n-2}{2n-5}$. 
    The equality holds if and only if $T \cong CG_{1, n-3; 1}$. 
    If the diameter of $T$ is $2$, then it is a star graph and the matching number is $1$, contradiction. 

    Suppose $m \geq 3$. 
    Note that $\sigma_2(Sp_{m-1,n-2m+1;2,1}) = \frac{1}{2}$ as $m - 1 \geq 2$.
    If the diameter of $T$ is at least $5$, then it contains the path graph $P_6$ as a subtree. 
    Then by~\cref{lem:monotone} $\sigma_2(T) \leq \sigma_2(P_6) = \frac{2}{5} < \frac{1}{2}$. 
    If the diameter of $T$ is $4$. 
    If $T$ contains the spider graph $Sp_{1,2; 3, 1}$ as a subtree, then by~\cref{lem:monotone} we have $\sigma_2(T) \leq \sigma_2(Sp_{1,2; 3, 1}) \approx 0.42 < \frac{1}{2}$. 
    Suppose $T$ does not contain the spider graph $Sp_{1,2;3, 1}$ as a subtree.
    Let $u$ and $v$ be two vertices of distance $4$. 
    This implies that neither of $u$ and $v$ is a twin vertex. 
    So the tree $T$ must be a spider graph $Sp_{p_1,p_2; 2, 1}$ with $2p_1+p_2+1 = n$ and $p_1 \geq 2$. 
    Since the matching number is $m$, it can only be the spider graph $Sp_{m-1, n-2m+1;2,1}$. 
    If the diameter of $T$ is at most $3$, then the matching number is at most $2$, contradiction. 
\end{proof}

Next we bound the matching number by the number of leaves and the diameter of the tree. 

\begin{lemma}\label{lem:visitor}
    Let $b \geq 3$ and $t \geq 1$ be integers. 
    Let $T$ be a tree with $b$ leaves and of diameter at most $4t$. 
    Then the matching number of $T$ is at most $bt$.
\end{lemma}

\begin{proof}
    Without loss of generality, we may assume the diameter is $4t$.
    Consider a diametral path of length $4t$. 
    It contributes at most $2t$ edges to the maximum matching. 
    For the rest of $(b-2)$ leaves, consider their distance to this diametral path. 
    The distance is at most $2t$, which would contribute at most $t$ edges to the maximum matching. 
    The total number of edges in the matching is at most $2t + (b-2)t = bt$. 
\end{proof}

\begin{lemma}\label{lem:recent}
    Let $b \geq 2$ and $t \geq 1$ be integers. 
    Let $T$ be a tree of diameter $4t+1$ with $b$ leaves and maximum matching number $bt+1$. 
    Then $T$ is a crab graph $CG_{b_1, b - b_1; 2t}$ for some positive integers $b_1 = 1,2, \ldots, b - 1$. 
\end{lemma}

\begin{proof}
    Consider a diametral path of length $4t+1$. 
    It contributes at most $2t+1$ edges to the maximum matching. 
    For the rest of $(b-2)$ leaves, consider their distance to this diametral path. 
    The distance is at most $2t$, which would contribute at most $t$ edges to the maximum matching. 
    The total matching is at most $2t+1 + (b-2)t = bt+1$. 
    The equality holds if and only if the diametral path contributes exactly $2t+1$ edges and the other paths contribute exactly $t$ edges each. 
    In such case, these leaves can only be connected to the middle two vertices of the diametral path and the distance from the leaves to the diametral path is exactly $2t$. 
    So the graph is a crab graph $CG_{b_1, b - b_1; 2t}$ for some $b_1 = 1,2, \ldots, b - 1$. 
\end{proof}

\begin{lemma}\label{lem:material}
    Let $b \geq 3$ and $t \geq 0$ be integers. 
    Let $T$ be a tree with $b$ leaves and of diameter at most $4t+3$. 
    Then the matching number of $T$ is at most $bt+2$.
\end{lemma}

\begin{proof}
    Without loss of generality, we may assume the diameter is $4t+3$.
    Consider a diametral path of length $4t+3$, say $v_0 \sim v_1 \sim \cdots v_{4t+3}$.
    It contributes at most $2t+2$ edges to the maximum matching. 
    We divide the cases by the number of edges it contributes to the maximum matching. 

    Suppose the diametral path contributes exactly $2t+2$ edges to the maximum matching. 
    Then these edges must be $v_0 v_1, v_2 v_3, \ldots, v_{4t+2} v_{4t+3}$. 
    We denote by $M_1$ the set of these edges. 
    For the rest of $(b-2)$ leaves, consider their distance to this diametral path. 
    The distance is at most $2t+1$, which would contribute at most $t+1$ edges to the maximum matching. 
    But each vertex on the diametral path has already been occupied by an edge in $M_1$. 
    So the path from each remaining leaf to the diametral path could contribute at most $t$ edges to the maximum matching. 
    The total number of edges in the matching is at most $2t+2 + (b-2)t = bt+2$. 

    Suppose the maximum matching does not contain the edges $v_{2t} v_{2t+1}$, $v_{2t+1} v_{2t+2}$ but contains the edge $v_{2t+2} v_{2t+3}$. 
    Then the diametral path contributes at most $2t+1$ edges to the maximum matching. 
    For the rest of $(b-2)$ leaves, consider their distance to this diametral path. 
    Suppose $b_1$ remaining leaves links to the diametral path through the vertex $v_{2t+1}$. 
    Then they contribute at most $b_1 t + 1$ edges to the maximum matching. 
    For the leaves links to the diametral path through the vertex $v_{2t+2}$, each of them contributes at most $t$ edges to the maximum matching.
    For the leaves links to the diametral path through other vertices, each of them contributes at most $t$ edges to the maximum matching since the distance is at most $2t$.
    The total number of edges in the matching is at most $2t+1 + b_1 t + 1 + (b - 2 - b_1)t = bt+2$. 
    
    Suppose the maximum matching does not contain the edges $v_{2t+1} v_{2t+2}$, $v_{2t+2} v_{2t+3}$ but contains the edge $v_{2t} v_{2t+1}$. 
    Similarly, the total number of edges in the matching is at most $bt+2$. 
    
    Suppose the maximum matching does not contain the edges $v_{2t+1} v_{2t+2}$, $v_{2t+2} v_{2t+3}$, or $v_{2t} v_{2t+1}$. 
    Then the diametral path contributes at most $2t$ edges to the maximum matching. 
    Suppose $b_1$ remaining leaves links to the diametral path through the vertex $v_{2t+1}$. 
    Then they contribute at most $b_1 t + 1$ edges to the maximum matching. 
    Suppose $b_2$ remaining leaves links to the diametral path through the vertex $v_{2t+2}$. 
    Then they contribute at most $b_2 t + 1$ edges to the maximum matching. 
    For the leaves links to the diametral path through other vertices, each of them contributes at most $t$ edges to the maximum matching since the distance is at most $2t$. 
    The total number of edges in the matching is at most $2t + b_1 t + 1+ b_2 t + 1+ (b - 2 - b_1 - b_2)t = bt+2$. 

    In any case, the matching number is at most $bt+2$.
\end{proof}

\begin{lemma}\label{lem:surface}
    Let $b \geq 3$ and $t \geq 1$ be integers. 
    Let $T$ be a tree with $b$ leaves and of diameter at most $4t+2$. 
    Then the matching number of $T$ is at most $bt+1$.
\end{lemma}

\begin{proof}
    Without loss of generality, we may assume the diameter is $4t+2$.
    Consider a diametral path of length $4t+2$, say $v_0 \sim v_1 \sim \cdots v_{4t+2}$.
    It contributes at most $2t+1$ edges to the maximum matching. 
    We divide the cases by the number of edges it contributes to the maximum matching. 

    Suppose the maximum matching contains one of the edges $v_{2t} v_{2t+1}$, $v_{2t+1} v_{2t+2}$. 
    For the rest of $(b-2)$ leaves, consider their distance to this diametral path. 
    Suppose $b_1$ remaining leaves links to the diametral path through the vertex $v_{2t+1}$. 
    Each of them contributes at most $t$ edges to the maximum matching since $v_{2t+1}$ has been occupied. 
    For the leaves links to the diametral path through other vertices, each of them contributes at most $t$ edges to the maximum matching since the distance is at most $2t$.
    The total number of edges in the matching is at most $2t+1 + b_1 t + (b - 2 - b_1)t = bt+1$. 

    Suppose the maximum matching does not contain the edges $v_{2t} v_{2t+1}$, $v_{2t+1} v_{2t+2}$. 
    Then the diametral path contributes at most $2t$ edges to the maximum matching. 
    For the rest of $(b-2)$ leaves, consider their distance to this diametral path. 
    Suppose $b_1$ remaining leaves links to the diametral path through the vertex $v_{2t+1}$. 
    Then they contribute at most $b_1 t + 1$ edges to the maximum matching. 
    For the leaves links to the diametral path through other vertices, each of them contributes at most $t$ edges to the maximum matching since the distance is at most $2t$.
    The total number of edges in the matching is at most $2t + b_1 t + 1 + (b - 2 - b_1)t = bt+1$. 
    
    In any case, the matching number is at most $bt+1$.
\end{proof}

We are prepared to prove~\cref{thm:fell,thm:older}.

\begin{proof}[{Proof of~\cref{thm:fell}}]
    Let $T$ be a tree with $b$ leaves and maximum matching number $m$. 

    If $m = 1$, then $T$ is a star and $\sigma_2(T) = 1$. 

    If $m = 2$, then $T$ is a crab graph $CG_{b_1, b - b_1; 1}$ for some $b_1 = 1,2,\ldots, b-1$. 
    We have $\sigma_2(T) = \frac{b}{b_1(b-b_1)+ b} \leq \frac{b}{2b-1}$. 
    The equality holds when $b_1 = 1$ or $b-1$. 

    If $m \geq 3$ and $b = 2$, then $T$ is the path graph $P_{2m}$ or $P_{2m+1}$. 
    We have $\sigma_2(T) = \frac{2}{2m-1}$ or $\sigma_2(T) = \frac{2}{2m}$. 

    Suppose $b \geq 3$ and $m = br+1, r \in \ZZ_+$. 
    By~\cref{lem:visitor}, the diameter of $T$ is at least $4r+1$. 
    If the diameter of $T$ is at least $4r+2$, then $\sigma_2(T) \leq \frac{2}{4r+2} = \frac{1}{2r+1} \leq \frac{1}{2r+\frac{b-1}{b}}$. 
    If the diameter of $T$ is exactly $4r+1$, then by~\cref{lem:recent} the tree $T$ is a crab graph $CG_{b_1, b-b_1; 2r}$ for some positive integer $b_1 = 1,2,\ldots, b-1$. 
    So $\sigma_2(T) = \frac{b}{b_1(b-b_1) + 2rb} \leq \frac{b}{2rb + b-1}$. 
    The equality holds if and only if $b_1 = 1$ or $b-1$. 
    
    Suppose $b \geq 3$ and $m = br+2, r \in \ZZ_+$. 
    By~\cref{lem:surface}, the diameter of $T$ is at least $4r+3$. 
    So $\sigma_2(T) \leq \frac{2}{4r+3}$.

    Suppose $b \geq 3$ and $m = br+s, 3 \leq s \leq b, r \in \ZZ_+$. 
    By~\cref{lem:material}, the diameter of $T$ is at least $4r+4$. 
    So $\sigma_2(T) \leq \frac{2}{4r+4} = \frac{1}{2r+2}$. 
\end{proof}

\begin{proof}[{Proof of~\cref{thm:older}}]
    Note that $\sigma_k(Sp_{1, b-1; 2m-1, 1}) = 1$ for $3 \leq k \leq b$. 
    Meanwhile, $\sigma_k(T) \leq 1$ for every tree $T$ (See~\cite[Theorem 1.5]{lin2024UpperBoundsSteklov} or~\cite[Corollary 2.12]{lin2025MaximizeSteklovEigenvalue}).
\end{proof}

\subsection{Algebraic connectivity}

    Let $M$ be a matrix. 
    We denote by $\Phi(M) = \Phi(M,x) = \det(xI - M)$ the characteristic polynomial of $M$. 
    The eigenvalues of $M$ are denoted by $\mu_1(M) \leq \mu_2(M) \leq \cdots$ if they are all real. 
    Let $G$ be a graph. 
    We denote by $L(G)$ the Laplacian matrix of $G$, by $\Phi(G) = \Phi(L(G))$ the characteristic polynomial of $L(G)$, and by $\lambda_k(G) = \lambda_k(L(G))$ the $k$-th Laplacian eigenvalue of $G$. 

We first show that shrinking a tree only enlarges the Laplacian eigenvalues. 

\begin{lemma}\label{lem:naturally}
    Let $T$ be a finite tree of order $n$. 
    Let $T'$ be a subtree of $T$ of order $n'$. 
    Then
    \begin{align}
        \lambda_{i}(T) \leq \lambda_{i}(T')
    \end{align}
    for $i = 1,2, \ldots, n'$.
\end{lemma}

\begin{proof}
    Let $\xi_1, \xi_2, \ldots, \xi_n$ be an orthogonal system of Laplacian eigenfunctions of $L(T)$ such that $\xi_i$ is a $\lambda_i(T)$-eigenfunction on $V(T)$ for $i = 1,2, \ldots, n$. 
    Let $\eta_1, \eta_2, \ldots, \eta_{n'}$ be an orthogonal system of Laplacian eigenfunctions of $L(T')$ such that $\eta_i$ is a $\lambda_i(T')$-eigenfunction on $V(T')$ for $i = 1,2, \ldots, n'$. 
    For a function $f$ on $V(T')$, we extend it to a function $\widetilde{f}$ on $V(T)$ such that $\widetilde{f}$ takes constant value in the same connected components of $T - E(T')$. 
    It is clear that $\lambda_1(T) = \lambda_1(T') = 0$. 
    Let $2 \leq i \leq n'$ be an integer. 
    Consider the function $\widetilde{f} = c_1 \widetilde{f}_1 + c_2 \widetilde{f}_2 + \cdots + c_i \widetilde{f}_i$, where $c_1, c_2, \ldots, c_i$ are constants that are not all zero such that $\lrangle{\widetilde{f}, \xi_j} = 0$ for every $j = 1,2, \ldots, i-1$. 
    Such choice of $c_1, c_2, \ldots, c_i$ exists because homogeneous linear system with $i-1$ equations on $i$ variables must have a nonzero solution. 
    Then
    \begin{align}
        \lambda_i(T) &\leq \frac{\sum_{uv \in E(T)}(\widetilde{f}(u) - \widetilde{f}(v))^2}{\sum_{v \in T} \widetilde{f}(v)^2} \\
        &\leq \frac{\sum_{uv \in E(T')}(f(u) - f(v))^2}{\sum_{v \in T'} f(v)^2} = \lambda_i(T'). \qedhere
    \end{align}
\end{proof}

\subsubsection{Laplacian eigenvalues of the crab graph}

In this subsection we determine the eigenvalues of $L(CG_{b_1, b_2;r})$. 

Note that we have an equitable partition of $CG_{b_1, b_2; r}$, namely the parts $V_{u, i}, V_{v, i}$ with $i = 0,1, \ldots, r$ given by
\begin{align}
    V_{u, 0} &= \set{u_0}, \\
    V_{u, i} &= \set{u_{i, 1}, \ldots, u_{i, b_1}}, 1 \leq i \leq r,  \\
    V_{v, 0} &= \set{v_0}, \\
    V_{v, i} &= \set{v_{i, 1}, \ldots, v_{i, b_2}}, 1 \leq i \leq r. 
\end{align} 

Let $\RR^V$ be the space of all real functions on $V$. 
Let $W$ be the subspace which takes constant value in each part. 
We first determine the eigenspace in $W^\perp$. 
It is obtained by lifting some eigenfunctions of $P_{2r+1}$ to $W^\perp$. 

Let $n$ be a positive integer. 
Recall that $L(P_n)$ is the Laplacian matrix of path graph with $n$ vertices. 
We denote by $B_{n-1}$ the matrix obtained by deleting one row and one column corresponding to one end of $P_n$ from $L(P_n)$. 
We denote by $H_{n-2}$ the matrix obtained by deleting two rows and two columns corresponding to two ends of $P_n$ from $L(P_n)$. 

\begin{lemma}[{\cite[Lemma 2.8]{guo2008ConjectureAlgebraicConnectivity}}]
    Set $\Phi(P_0) = 0, \Phi(B_0) = 1$, and $\Phi(H_0) = 1$. 
    Then we have
    \begin{enumerate}
        \item $\Phi(P_{n+1}) = (x-2) \Phi(P_n) - \Phi(P_{n-1})\ (n \geq 1)$; 
        \item $x \Phi(B_n) = \Phi(P_{n+1}) + \Phi(P_n)$; 
        \item $x \Phi(H_n) = \Phi(P_{n+1})\ (n \geq 1)$;
        \item $\Phi(P_m) \Phi(P_n) - \Phi(P_{m-1}) \Phi(P_{n+1}) = \Phi(P_{m-1}) \Phi(P_{n-1}) - \Phi(P_{m-2}) \Phi(P_{n})\ (m \geq 2, n \geq 1)$.
    \end{enumerate}
\end{lemma}

\begin{corollary}
    $\Phi(B_n) + \Phi(B_{n-1}) = \Phi(P_n)$.
\end{corollary}

\begin{proof}
    \begin{align}
        x(\Phi(B_n) + \Phi(B_{n-1})) &= \Phi(P_{n+1}) + \Phi(P_n) + \Phi(P_n) + \Phi(P_{n-1}) \\
        &= (x-2) \Phi(P_n) - \Phi(P_{n-1}) + 2 \Phi(P_n) + \Phi(P_{n-1}) \\
        &= x \Phi(P_n). \qedhere
    \end{align}
\end{proof}

It is well-known that the roots of $\Phi(P_{n})$ are $4\sin^2 \frac{(i-1)\pi}{2n}$ for $i = 1,2, \ldots, n$, and the roots of $\Phi(B_{n})$ are $4 \sin^2 \frac{(2i-1)\pi}{4n+2}$ for $i = 1,2, \ldots, n$. 

\begin{lemma}\label{lem:writer}
    Let $\lambda = \lambda_{2i}(P_{2r+1})$ be the $(2i)$-th eigenvalue of $L(P_{2r+1})$ with $i = 1,2, \ldots, r$. 
    Then $W^\perp$ contains a $\lambda$-eigenspace of $L(CG_{b_1, b_2; r})$ with dimension $b_1 + b_2 - 2$.
\end{lemma}

\begin{proof}
    The $\lambda_i$-eigenfunction on $P_{2r+1} = w_1 \sim w_2 \sim \cdots \sim w_{2r+1}$ is given by $\xi(w_j) = (-1)^{j-1} (1 - 2 \cos \frac{(i-1)\pi}{n} + 2 \cos \frac{2(i-1)\pi}{n} - 2 \cos \frac{3(i-1)\pi}{n} + \cdots + (-1)^{j-1} 2 \cos \frac{(j-1) (i-1)\pi}{n})$ for $j = 1,2, \ldots, 2r+1$. 
    Note that $\xi(w_{r+1}) = 0$ for $\lambda = \lambda_{2i}(P_{2r+1})$ with $i = 1,2, \ldots, r$. 
    Consider the path $u_{m, r} \sim u_{m, r-1} \sim \cdots u_{m, 1} \sim u_{0} \sim u_{m+1, 1} \sim \cdots \sim u_{m+1, r}$ for $m = 1,2, \ldots, b_1 -1$. 
    We extend the $\lambda$-eigenfunction on $P_{2r+1}$ to $CG_{b_1, b_2; r}$ by assigning $0$ on the rest of vertices. 
    This gives us $b_1 - 1$ linearly independent $\lambda$-eigenfunctions of $L(CG_{b_1, b_2; r})$ in $W^\perp$. 
    Similarly, we can consider the path $v_{m, r} \sim v_{m, r-1} \sim \cdots v_{m, 1} \sim v_{0} \sim v_{m+1, 1} \sim \cdots \sim v_{m+1, r}$ for $m = 1,2, \ldots, b_2 -1$. 
    In total, we get $(b_1 + b_2 - 2)$ linear independent $\lambda$-eigenfunctions of $L(CG_{b_1, b_2; r})$ in $W^\perp$. 
\end{proof}

Next we determine the eigenfunction in $W$. 

Consider the quotient matrix $Q_{b_1, b_2; r}$ of $CG_{b_1, b_2; r}$ with parts $V_{u, i}, V_{v, i}$ given above. 
Define
\begin{align}
    Q_{b_1, b_2; r} = 
    \begin{bmatrix}
        b_1 + 1 & -b_1 \bm{e}_1^\top & -1 & \bm{0} \\
        -\bm{e}_1 & B_{r} & \bm{0} & \bm{0} \\
        -1 & \bm{0} & b_2 + 1 & -b_2 \bm{e}_1^\top \\
        \bm{0} & \bm{0} & -\bm{e}_1 & B_{r}
    \end{bmatrix},
\end{align}
where $\bm{e}_1$ is the column vector with the only nonzero entry $1$ at the first entry. 
Then $L(CG_{b_1, b_2; r}) P = P Q_{b_1, b_2; r}$, where $P$ consists of characteristic columns given by
\begin{align}
    P = 
    \begin{bmatrix}
        \bm{1}_{V_{u, 0}} & \bm{1}_{V_{u, 1}} & \cdots & \bm{1}_{V_{u, r}} & \bm{1}_{V_{v, 0}} & \bm{1}_{V_{v, 1}} & \cdots & \bm{1}_{V_{v, r}}
    \end{bmatrix}.
\end{align}

\begin{lemma}
    Let $\xi$ be a $\mu$-eigenfunction of $Q_{b_1, b_2; r}$. 
    Then $\xi$ induces a $\mu$-eigenfunction of $L(CG_{b_1, b_2; r})$ in $W$. 
\end{lemma}

\begin{proof}
    Note that $L(CG_{b_1, b_2; r}) P \xi = P Q_{b_1, b_2; r} \xi = \mu P \xi$. 
\end{proof}

The second-smallest eigenvalue of $CG_{b_1, b_2; r}$ must come from $Q_{b_1, b_2; r}$.

\begin{lemma}
    $\lambda_2(CG_{b_1, b_2; r}) = \mu_2(Q_{b_1, b_2; r}) \leq \lambda_2(P_{2r+2})$.
\end{lemma}

\begin{proof}
    Note that $W^\perp$ contains $r(b_1 + b_2 -2)$ linearly independent eigenfunctions, and $W$ contains $2r+2$ linearly independent eigenfunctions. 
    In total, we have $r(b_1 + b_2) + 2$ linearly independent eigenfunctions. 
    So we have determined all eigenvalues, which comes from eigenvalues of $L(P_{2r+1})$ and $Q_{b_1, b_2; r}$. 
    Since $P_{2r+2}$ is a subtree of $CG_{b_1, b_2; r}$, we have $\lambda_2(CG_{b_1, b_2; r}) \leq \lambda_2(P_{2r+2}) < \lambda_2(P_{2r+1})$. 
    Therefore, $\lambda_2(CG_{b_1, b_2; r}) = \mu_2(Q_{b_1, b_2; r})$.
\end{proof}

Next we compare the algebraic connectivity of $CG_{b_1, b_2; r}$ with fixed leaf number $b = b_1 + b_2$.

\begin{lemma}
    \begin{align}
        \Phi(Q_{b_1, b_2; r}) &= (x^2 - (b_1 + b_2 + 2)x + b_1 b_2 + b_1 + b_2) \Phi(B_{r}) \Phi(B_r) \\
        &-((b_1 + b_2)x - 2b_1 b_2 - b_1 - b_2) \Phi(B_r) \Phi(B_{r-1}) \\
        &+ b_1 b_2 \Phi(B_{r-1}) \Phi(B_{r-1}).
    \end{align}
\end{lemma}

\begin{proof}
    By Laplace expansion, we have
    \begin{align}
        \Phi(Q_{b_1, b_2; r}) =& (x - b_1 - 1)(x - b_2 - 1) \Phi(B_{r}) \Phi(B_{r}) \\
        & -(x - b_1 - 1)b_2 \Phi(B_{r-1}) \Phi(B_{r}) \\
        & -(x - b_2 - 1)b_1 \Phi(B_{r}) \Phi(B_{r-1}) \\
        & +b_1 b_2 \Phi(B_{r-1}) \Phi(B_{r-1}) \\
        =& (x^2 - (b_1 + b_2 + 2)x + b_1 b_2 + b_1 + b_2) \Phi(B_{r}) \Phi(B_r) \\
        & -((b_1 + b_2)x - 2b_1 b_2 - b_1 - b_2) \Phi(B_r) \Phi(B_{r-1}) \\
        & + b_1 b_2 \Phi(B_{r-1}) \Phi(B_{r-1}). \qedhere
    \end{align}
\end{proof}

\begin{corollary}
    \begin{align}
        \Phi(Q_{b_1, b_2; 1}) &= x(x^3 - (b1+b_2+4) x^2 + (b_1 b_2 + 2b_1 + 2b_2 + 5) x - (b_1 + b_2 + 2)).
    \end{align}
\end{corollary}

\begin{corollary}
    \begin{align}
        \Phi(Q_{1, b-1; 1}) &= x(x^3 - (b+4) x^2 + (3b+4) x - b - 2).
    \end{align}
    Hence, $\lambda_2(CG_{1, b-1; 1})$ is equal to the smallest root of $x^3 - (b+4) x^2 + (3b+4) x - b - 2 = 0$.
\end{corollary}

\begin{lemma}
    Let $1 \leq b_1 \leq b_2 - 1$ be integers. 
    Then
    \begin{align}
        \Phi(Q_{b_1 + 1, b_2 - 1; r}) - \Phi(Q_{b_1, b_2; r}) &= (b_2 - b_1 - 1) (\Phi(B_{r-1}) + \Phi(B_r))^2 \\
        &= (b_2 - b_1 - 1) \Phi^2 (P_r) \geq 0.
    \end{align}
\end{lemma}

\begin{proof}
    It follows from direct computation. 
\end{proof}

\begin{corollary}
    Let $1 \leq b_1 \leq b_2 - 1$ be integers. 
    Then 
    $\mu_2(Q_{b_1 + 1, b_2 - 1; r}) \leq \mu_2(Q_{b_1, b_2; r})$. 
\end{corollary}

\begin{proof} 
    Note that the order of $\Phi(Q_{b_1+1, b_2-1; r})$ is $2r+2$ (even). 
    Since $0$ is a simple eigenvalue for $Q_{b_1+1, b_2-1; r}$, we have $\Phi(Q_{b_1+1, b_2-1; r}, \epsilon) < 0$ for $\epsilon > 0$ small enough. 
    Let $\mu = \mu_2(Q_{b_1, b_2; r}) > 0$ be the second-smallest eigenvalue of $Q_{b_1, b_2; r}$. 
    Then $\Phi(Q_{b_1 + 1, b_2 - 1; r}, \mu) = \Phi(Q_{b_1 + 1, b_2 - 1; r}, \mu) - \Phi(Q_{b_1, b_2; r}, \mu) \geq 0$ implies that $\mu_2(Q_{b_1 + 1, b_2 - 1; r}) \leq \mu = \mu_2(Q_{b_1, b_2; r})$.
\end{proof}

\begin{corollary}
    Let $1 \leq b_1 \leq b_2 - 1$ be integers. 
    Then 
    $\lambda_2(CG_{b_1 + 1, b_2 - 1; r}) \leq \lambda_2(CG_{b_1, b_2; r})$. 
\end{corollary}

In fact, the inequality is strict if $b_2 > b_1 + 1$. 

\begin{lemma}
    Let $1 \leq b_1 < b_2 - 1$ be integers. 
    Then 
    $\lambda_2(CG_{b_1 + 1, b_2 - 1; r}) < \lambda_2(CG_{b_1, b_2; r})$. 
\end{lemma}

\begin{proof}
    Suppose $\lambda = \lambda_2(CG_{b_1 + 1, b_2 - 1; 1}) = \lambda_2(CG_{b_1, b_2; 1})$. 
    Then $(b_2 - b_1 - 1) \Phi^2 (P_r, \lambda) = 0$. 
    Hence, $\Phi(P_r, \lambda) = 0$.
    But $0 < \lambda \leq \lambda_2(P_{2r+2}) < \lambda_2(P_r)$. 
    Contradiction. 
\end{proof}

\begin{corollary}
    Let $1 \leq b_1 < b$ be integers. 
    Then 
    $\lambda_2(CG_{b_1, b - b_1; r}) \leq \lambda_2(CG_{1, b - 1; r})$. 
    The equality holds if and only if $b_1 = 1$ or $b-1$. 
\end{corollary}

In the following we give a lower bound of $\lambda_2(CG_{1, b-1; r})$.

\begin{lemma}
    Let $b \geq 2$ be an integer. 
    Then
    \begin{align}
        \Phi(Q_{1, b - 1; r}) - \Phi(Q_{1, 1; r}) &= (b-2)((1-x)\Phi(B_r) + \Phi(P_r)) \Phi(P_r).
    \end{align}
\end{lemma}

\begin{proof}
    Note that 
    \begin{align}
        \Phi(Q_{1, b - 1; r}) =& (x^2 - (b + 2)x + 2b-1) \Phi(B_{r}) \Phi(B_r) \\
        &-(bx - 3b + 2) \Phi(B_r) \Phi(B_{r-1}) \\
        &+ (b-1) \Phi(B_{r-1}) \Phi(B_{r-1}).
    \end{align}
    Hence,
    \begin{align}
        \Phi(Q_{1, b - 1; r}) - \Phi(Q_{1, 1; r}) =& (b-2)(-(x-2) \Phi(B_r) \Phi(B_r) \\
        &-(x-3) \Phi(B_r) \Phi(B_{r-1}) + \Phi(B_{r-1}) \Phi(B_{r-1})) \\
        =& (b-2)( (2-x) \Phi(B_r) \Phi(P_r) + \Phi(B_{r-1}) \Phi(P_r)) \\
        =& (b-2)( (1-x) \Phi(B_r) + \Phi(P_r)) \Phi(P_r).  \qedhere
    \end{align}
\end{proof}

\begin{lemma}
    Let $b \geq 3$ be an integer. 
    Then
    $\lambda_2(CG_{1, b-1; r}) > \lambda_2(P_{2r+3})$. 
\end{lemma}

\begin{proof}
    Note that $Q_{1,1, r} = L(P_{2r+2})$. 
    It is clear that $\Phi(Q_{1, b - 1; r}, 0) = \Phi(Q_{1, 1; r}, 0) = 0$. 
    Consider $0 < \theta \leq \lambda_2(P_{2r+3}) < \lambda_2(P_{2r+1}) = \lambda_2(B_r)$. 
    We have $\Phi(Q_{1,1; r}, \theta) < 0$. 
    Meanwhile, $( (1-x) \Phi(B_r, \theta) + \Phi(P_r, \theta)) \Phi(P_r, \theta) < 0$. 
    Therefore, $\Phi(Q_{1, b-1; r}, \theta) < 0$. 
    This implies that $\lambda_2(Q_{1, b-1; r}) > \lambda_2(P_{2r+3})$.
    So $\lambda_2(CG_{1, b-1; r}) > \lambda_2(P_{2r+3})$.
\end{proof}

\subsubsection{Laplacian eigenvalues of the spider graph}

In this subsection we determine the Laplacian eigenvalues of $L(Sp_{b ; r})$. 

Note that we have an equitable partition of $Sp_{b ; r}$, namely the parts $V_{u, i}$ with $i = 0,1, \ldots, r$ given by
\begin{align}
    V_{u, 0} &= \set{u_0}, \\
    V_{u, i} &= \set{u_{i, 1}, \ldots, u_{i, b}}, 1 \leq i \leq r. 
\end{align} 

Let $\RR^V$ be the space of all real functions on $V$. 
Let $W$ be the subspace which takes constant value in each parts. 
We first determine the eigenspace in $W^\perp$. 
It is obtained by lifting some eigenfunctions of $P_{2r+1}$ to $W^\perp$. 

\begin{lemma}
    Let $b \geq 2$ be an integer. 
    Let $\lambda = \lambda_{2i}(P_{2r+1})$ be the $(2i)$-th eigenvalue of $L(P_{2r+1})$ with $i = 1,2, \ldots, r$. 
    Then $W^\perp$ contains a $\lambda$-eigenspace of $L(Sp_{b ; r})$ with dimension $b - 1$.
\end{lemma}

\begin{proof}
    The proof is identical to the proof of~\cref{lem:writer}.
\end{proof}

Next we determine the eigenfunctions in $W$.

Consider the quotient matrix $Q_{b; r}$ of $Sp_{b ; r}$ with parts $V_{u, i}$ given above. 
Define
\begin{align}
    Q_{b; r} = 
    \begin{bmatrix}
        b & -b_1 \bm{e}_1^\top \\
        -\bm{e}_1 & B_{r}
    \end{bmatrix}.
\end{align}
Then $L(CG_{b_1, b_2; r}) P = P Q_{b_1, b_2; r}$, where $P$ consists of characteristic columns given by
\begin{align}
    P = 
    \begin{bmatrix}
        \bm{1}_{V_{u, 0}} & \bm{1}_{V_{u, 1}} & \cdots & \bm{1}_{V_{u, r}}
    \end{bmatrix}.
\end{align}

\begin{lemma}
    Let $\xi$ be a $\mu$-eigenfunction of $Q_{b; r}$. 
    Then $\xi$ induces a $\mu$-eigenfunction of $L(Sp_{b ; r})$ in $W$. 
\end{lemma}

\begin{proof}
    Note that $L(Sp_{b; r}) P \xi = P Q_{b; r} \xi = \mu P \xi$. 
\end{proof}

Next we give a lower bound of $\lambda_2(Q_{b; r})$.

\begin{lemma}
    Let $1 \leq k \leq n$ be integers. 
    Let $P$ be a positive definite matrix of order $n$. 
    Let $S$ be a positive semidefinite matrix of order $n$. 
    Then $\mu_k(PS) \geq \mu_1(P) \mu_k(S)$.
\end{lemma}

\begin{proof}
    \begin{align}
        \mu_k(PS) &= \mu_k(\sqrt{P} S \sqrt{P}) \\
        &= \min_{\dim U = k} \max_{0 \neq x \in U} \frac{x^\top \sqrt{P} S \sqrt{P}x}{x^\top x} \\
        &= \min_{\dim U = k} \max_{0 \neq x \in U} \frac{(\sqrt{P}x)^\top S (\sqrt{P}x)}{(\sqrt{P}x)^\top (\sqrt{P}x)} \cdot \frac{x^\top P x}{x^\top x} \\
        &\geq \mu_1(P) \min_{\dim U = k} \max_{0 \neq x \in U} \frac{(\sqrt{P}x)^\top S (\sqrt{P}x)}{(\sqrt{P}x)^\top (\sqrt{P}x)} \\
        &= \mu_1(P) \mu_k(S). \qedhere
    \end{align}
\end{proof}

\begin{corollary}
    Let $b \geq 2$ be an integer. 
    Then 
    $\mu_2(Q_{b; r}) \geq \mu_2(Q_{1; r}) = \lambda_2(P_{r+1})$. 
\end{corollary}

\begin{proof}
    Note that $Q_{b; r} = \Lambda Q_{1; r}$ where $\Lambda$ is the diagonal matrix $\diag(b, 1, \ldots, 1)$. 
\end{proof}

The second-smallest Laplacian eigenvalue of $Sp_{b; r}$ must come from $P_{2r+1}$.

\begin{lemma}
    Let $b \geq 2, r \geq 1$ be integers. 
    Then $\lambda_2(Sp_{b ; r}) = \lambda_2(P_{2r+1}) = 4 \sin^2 \frac{\pi}{4r+2}$.
\end{lemma}

\begin{proof}
    Note that $W^\perp$ contains $r(b - 1)$ linearly independent eigenfunctions, and $W$ contains $r+1$ linearly independent eigenfunctions. 
    In total, we have $rb + 1$ linearly independent eigenfunctions. 
    So we have determined all eigenvalues, which comes from eigenvalues of $L(P_{2r+1})$ and $Q_{b; r}$. 
    Since $\lambda_2(Q_{b; r}) \geq \lambda_2(P_{r+1}) > \lambda_2(P_{2r+1})$, we have $\lambda_2(Sp_{b ; r}) = \lambda_2(P_{2r+1})$. 
\end{proof}

\subsubsection{Proofs for the maximum algebraic connectivity}

\begin{proof}[Proof of~\cref{thm:ranch}]
    Let $T$ be a tree with $n$ vertices and matching number $m$. 
    If $T$ contains a path of length $\ell$, then the matching number of $T$ is at least $\ceil{\ell/2}$. 
    Therefore, if the matching number of a tree is $m$, then the diameter $D$ of the tree is at most $2m$. 

    If $m = 1$, then the diameter is at most $2$. 
    The only tree with diameter $1$ is the star graph $S_2$. 
    Every tree with diameter $2$ is a star graph. 
    It is clear that the second Laplacian eigenvalue of a star graph is $1$. 

    If $m = 2$, then the diameter is at most $4$. 
    If the diameter of $T$ is $4$, then it contains the path graph $P_5$ as a subtree. 
    Then by~\cref{lem:naturally} $\lambda_2(T) \leq \lambda_2(P_5) < \lambda_2(CG_{1, n-3; 1})$. 
    If the diameter of $T$ is $3$, then $T$ is a crab graph $CG_{p,q;1}$ with $p + q = n-2$. 
    Then $\lambda_2(T) = \lambda_2(CG_{p, q; 1}) \leq \lambda_2(CG_{1, n-3; 1})$. 
    The equality holds if and only if $T \cong CG_{1, n-3; 1}$. 
    If the diameter of $T$ is $2$, then it is a star graph and the matching number is $1$, contradiction. 

    Suppose $m \geq 3$. 
    Note that $\lambda_2(Sp_{m-1,n-2m+1;2,1}) \geq \lambda_2(Sp_{n-m; 2}) = \lambda_2(P_{5})$ as $n-m \geq 2$.
    If the diameter of $T$ is at least $5$, then it contains the path graph $P_6$ as a subtree. 
    Then by~\cref{lem:naturally} $\lambda_2(T) \leq \lambda_2(P_6) < \lambda_2(P_5)$. 
    If the diameter of $T$ is $4$. 
    If $T$ contains the spider graph $Sp_{1, 2; 3, 1}$ as a subtree, then by~\cref{lem:naturally} we have $\lambda_2(T) \leq \lambda_2(Sp_{1,2;3,1}) \approx 0.32 < \lambda_2(P_5) \approx 0.38$. 
    Suppose $T$ does not contain the spider graph $Sp_{1,2; 3, 1}$ as a subtree.
    Let $u$ and $v$ be two vertices of distance $4$. 
    This implies that neither of $u$ and $v$ is a twin vertex. 
    So the tree $T$ must be a spider graph $Sp_{p_1,p_2; 2, 1}$ with $2p_1+p_2+1 = n$ and $p_1 \geq 2$. 
    Since the matching number is $m$, it can only be the spider graph $Sp_{m-1, n-2m+1;2,1}$. 
    If the diameter of $T$ is at most $3$, then the matching number is at most $2$, contradiction. 
\end{proof}

\begin{proof}[{Proof of~\cref{thm:unit}}]
    Let $T$ be a tree with $b$ leaves and maximum matching number $m$. 

    If $m = 1$, then $T$ is a star and $\lambda_2(T) = 1$. 

    If $m = 2$, then $T$ is a crab graph $CG_{b_1, b - b_1; 1}$ for some $b_1 = 1,2,\ldots, b-1$. 
    We have $\lambda_2(T) = \lambda_2(CG_{b_1, b-b_1; 1}) \leq \lambda_2(CG_{1, b-1; 1})$. 
    The equality holds if and only if $T$ is the crab graph $CG_{1, b-1; 1}$.

    If $m \geq 3$ and $b = 2$, then $T$ is the path graph $P_{2m}$ or $P_{2m+1}$. 
    We have $\sigma_2(T) = \lambda_2(P_{2m})$ or $\sigma_2(T) = \lambda_2(P_{2m+1})$. 

    Suppose $b \geq 3$ and $m = br+1, r \in \ZZ_+$. 
    By~\cref{lem:visitor}, the diameter of $T$ is at least $4r+1$. 
    If the diameter of $T$ is at least $4r+2$, then $\lambda_2(T) \leq \lambda_2(P_{4r+3}) < \lambda_2(CG_{1, b-1; 2r})$. 
    If the diameter of $T$ is exactly $4r+1$, then by~\cref{lem:recent} the tree $T$ is a crab graph $CG_{b_1, b-b_1; 2r}$ for some positive integer $b_1 = 1,2,\ldots, b-1$. 
    So $\lambda_2(T) = \lambda_2(CG_{b_1, b-b_1; 2r}) \leq \lambda_2(CG_{1, b-1; 2r})$. 
    The equality holds if and only if $T$ is the crab graph $CG_{1, b-1; 2r}$.
    
    Suppose $b \geq 3$ and $m = br+2, r \in \ZZ_+$. 
    By~\cref{lem:surface}, the diameter of $T$ is at least $4r+3$. 
    So $\lambda_2(T) \leq \lambda_2(P_{4r+4})$.

    Suppose $b \geq 3$ and $m = br+s, 3 \leq s \leq b, r \in \ZZ_+$. 
    By~\cref{lem:material}, the diameter of $T$ is at least $4r+4$. 
    So $\lambda_2(T) \leq \lambda_2(P_{4r+5})$. 
\end{proof}

\section*{Acknowledgements}
Huiqiu Lin was supported by the National Natural Science Foundation of China (No. 12271162, No. 12326372), and the Natural Science Foundation of Shanghai (No. 22ZR1416300 and No. 23JC1401500) and The Program for Professor of Special Appointment (Eastern Scholar) at Shanghai Institutions of Higher Learning (No. TP2022031). 
Da Zhao was supported in part by the National Natural Science Foundation of China (No. 12471324), the Natural Science Foundation of Shanghai, Shanghai Sailing Program (No. 24YF2709000), and the Fundamental Research Funds for the Central Universities.

\bibliographystyle{plain}
\bibliography{ref}

\end{document}